\documentclass[preprint,12pt]{elsarticle}
\usepackage{amsthm,amsmath,amssymb}
\usepackage[colorlinks=true,citecolor=black,linkcolor=black,urlcolor=blue]{hyperref}
\usepackage{graphicx}
\usepackage{float}
\usepackage{enumitem}

\newcommand{\arxiv}[1]{\href{http://arxiv.org/abs/#1}{\texttt{arXiv:#1}}}

\theoremstyle{plain}
\newtheorem{theorem}{Theorem}
\newtheorem{lemma}[theorem]{Lemma}
\newtheorem{corollary}[theorem]{Corollary}
\newtheorem{proposition}[theorem]{Proposition}

\theoremstyle{definition}
\newtheorem{definition}[theorem]{Definition}

\newtheorem{problem}[theorem]{Problem}

\newtheorem{remark}[theorem]{Remark}

\newcommand{\F}{\mathcal{F}}

\newcommand{\GF}{\mathrm{GF}}

\title{Partial-twuality polynomial interpolation for binary delta-matroids}

\author{Zhuo Li$^{1}$, Qi Yan$^{2}$\footnote{Corresponding author.}, Xian'an Jin$^{1,3}$\\
        ~\\
        \small $^{1}$School of Mathematical Sciences, Xiamen University, P. R. China\\
		\small $^{2}$School of Mathematics and Statistics, Lanzhou University, P. R. China\\
        \small $^{3}$School of Mathematics and Statistics, Qinghai Minzu University, P. R. China\\
    \small\tt Email:   lzhuo@stu.xmu.edu.cn, yanq@lzu.edu.cn, xajin@xmu.edu.cn}

\date{}
\journal{~~}

\begin{document}
\begin{abstract}
Gross, Mansour and Tucker introduced the partial-twuality polynomials for ribbon graphs and investigated the interpolation property of these polynomials. The ribbon group generated by $\delta$ and $\tau$ acts on set systems as twist $\ast$ and loop complementation $\times$, yielding five nontrivial twuality operators:
\(\{\ast,\times,\ast\times  ,\times\ast ,\ast\times\ast \}.\)
Yan and Jin extended partial-twuality polynomials to set systems, yielding partial-$\bullet$ polynomials with $\bullet\in\{\ast,\times,\ast \times  ,\times\ast ,\ast \times  \ast \}.$
For partial-$\ast$ polynomials, Zhao and Yan proved that this polynomial is either even, odd, or both even-interpolating and odd-interpolating for every binary delta-matroid.
In this paper, we extend this interpolation property to all the remaining nontrivial partial-twualities of binary delta-matroids.
Consequently, for every binary delta-matroid and every
$\bullet\in\{\ast,\times,\ast \times  ,\times\ast ,\ast \times  \ast \}$, the partial-$\bullet$
polynomial is either even, odd, or both even-interpolating and odd-interpolating. We also provide examples to show that the binary assumption is essential.
\end{abstract}

\maketitle

\section{Introduction}
In \cite{Wilson1979}, Wilson studied the ribbon group, which is generated by geometric duality $\delta$ and Petrie duality $\tau$ and yields five nontrivial twuality operators:
\(
\{\delta, \tau, \delta\tau, \tau\delta, \delta\tau\delta\}.
\)
Abrams and Ellis-Monaghan \cite{Abrams2022} referred to these five operators as twualities. Chmutov \cite{Chmutov2009} introduced the partial duality of ribbon graphs, a generalization of geometric duality. Ellis-Monaghan and Moffatt \cite{EllisMonaghanMoffatt2012} generalized this partial duality construction to the other four operators, which they called partial-twualities.

Gross, Mansour, and Tucker \cite{GrossMansourTucker2021} introduced partial-twuality polynomials for ribbon graphs, which enumerate all partial twuals by Euler genus. They showed that for every orientable ribbon graph, the partial-dual polynomial is an even polynomial, while for every ribbon graph, the partial-$\tau$ polynomial and the partial-$\delta\tau\delta$ polynomial are both interpolating.

Delta-matroids, introduced by Bouchet in \cite{Bouchet1987}, provide an algebraic framework for ribbon graphs. In \cite{ChunMoffattNobleRueckriemen2019} and \cite{Chun2019F}, Chun et al. showed that under the correspondence between ribbon graphs and delta-matroids, partial duality corresponds to the twist, while partial Petrial corresponds to loop complementation. Then the ribbon group acts on set systems via the two operations $\ast$ and $\times$, where $\ast$ denotes twist and $\times$ denotes loop complementation. The analogues of the five ribbon graph twualities are
\(
\{\ast, \times, \ast \times , \times \ast, \ast \times  \ast\}.
\)

For a set system $D=(E,\F)$ and $\bullet\in\{\ast, \times, \ast \times , \times \ast, \ast \times \ast\}$, the partial-$\bullet$ polynomial, introduced by Yan and Jin in \cite{YanJin2024}, is
\[
^{\partial}w_{D}^{\bullet}(z)
=
\sum_{A\subseteq E} z^{w(D^{\bullet \mid A})},
\]
where $w(D)$ is the width of \(D\). When $\bullet = \ast$, this is the twist polynomial \cite{YanJin2022}.
It is natural to investigate the interpolation property of such polynomials. Yan and Jin \cite{YanJin2024} proved that, for vf-safe delta-matroids, the partial-$\times$ and partial-$\ast \times \ast$ polynomials are interpolating, and all partial-twuality polynomials have no gaps of size at least two. However, for the other three nontrivial partial-twualities $\{\ast$, $\ast\times$, $\times\ast\}$, the polynomials may not be interpolating. Gross, Mansour and Tucker \cite{GrossMansourTucker2020} asked the following problem for non-orientable ribbon graphs.
\begin{problem}[\cite{GrossMansourTucker2020}]\label{pro1}
    Under what interesting sets of sufficient conditions on a non-orientable ribbon graph is the partial-dual polynomial even-interpolating, odd-interpolating, or both?
\end{problem}

Recently, Zhao and Yan \cite{ZhaoYan2026} proved that for every binary delta-matroid $D$, the twist polynomial ${}^{\partial}w_{D}^{\ast}(z)$ is either even, odd, or both even-interpolating and
odd-interpolating, thus answering Problem \ref{pro1}.
We similarly consider this problem for the remaining partial-twuality polynomials. The challenging cases are the two non-involutive operations $\ast \times$ and $\times \ast$, which correspond respectively to $\delta\tau$ and $\tau\delta$ for ribbon graphs. Our main result is the following.
\begin{theorem}\label{thm}
Let $D$ be a binary delta-matroid and let
\(\bullet\in\{\ast,\times,\ast \times  ,\times\ast ,\ast\times\ast \}.\)
Then the partial-$\bullet$ polynomial
\(^{\partial}w_{D}^{\bullet}(z)\)
is either an even polynomial, an odd polynomial, or both even-interpolating and
odd-interpolating.
\end{theorem}

Theorem \ref{thm} is not true for non-binary delta-matroids. We construct two non-binary delta-matroids such that the twist polynomial of the first, and the partial-$\times\ast$ and partial-$\ast\times$ polynomials of the second each contain non-zero terms in both even and odd degrees, but it is not both even-interpolating and odd-interpolating.

The remainder of this paper is structured as follows. Section \ref{2} recalls relevant definitions concerning set systems and delta-matroids. In Section \ref{3}, we establish several auxiliary lemmas. In Section \ref{4}, we state our main results and give the proof of Theorem \ref{thm}. In Section \ref{5}, we transfer the results on delta-matroids to ribbon graphs. Section \ref{6} presents examples demonstrating that the binary condition in Theorem \ref{thm} is necessary.

\section{Preliminaries}\label{2}

\subsection{Set systems and delta-matroids}

A \emph{set system} is a pair \(D=(E,\F),\)
where $E$ is a finite set and $\F\subseteq 2^E$ is a collection of \emph{feasible sets}.  We often use $\F(D)$ to denote the set of feasible sets of $D$. The set
system is \emph{proper} if $\F\neq\emptyset$. For $X,Y\subseteq E$, let
\(
X\triangle Y=(X\cup Y)\backslash(X\cap Y)
\)
denote symmetric difference. Throughout the paper, we often omit set brackets for singletons.

\begin{definition}[\cite{Bouchet1987}]
A \emph{delta-matroid} is a proper set system $D=(E,\F)$ satisfying the symmetric exchange
axiom: for any $X,Y\in\F$ and any $u\in X\triangle Y$, there exists
$v\in X\triangle Y$ (possibly $v=u$) such that
\(
X\triangle\{u,v\}\in\F.
\)
\end{definition}

Let $D=(E,\F)$ be a proper set system and let $\F_{\max}(D)$ and $\F_{\min}(D)$ be the collections of feasible sets of maximum
and minimum cardinality, respectively. Write
\[
r_{\max}(D)=\max\{|F| \mid F\in\F\},
\quad
r_{\min}(D)=\min\{|F|\mid F\in\F\}.
\]
The \emph{width} of $D$ is
\[
w(D)=r_{\max}(D)-r_{\min}(D).
\]
For $0\leq i\leq w(D)$, define
\[
\F_{\min+i}(D)=\{F\in\F\mid |F|=r_{\min}(D)+i\},
\]
and
\[
\F_{\max-i}(D)=\{F\in\F\mid |F|=r_{\max}(D)-i\}.
\]
Aside from $\mathcal{F}_{\max}(D)$ and $\mathcal{F}_{\min}(D)$, all other collections may be empty.
A set system  $D=(E,\F)$  is \emph{even} if all feasible sets have the same parity of cardinality.
Equivalently, \(|F_1\triangle F_2|\)
is even for all $F_1,F_2\in\F$.

\subsection{Deletion and contraction}
For a delta-matroid $D = (E, \F)$, and $e \in E$, if $e$ is in every feasible set of $D$, then we say that $e$ is a \emph{coloop} of $D$. If $e$ is in no feasible set of $D$, then we say that $e$ is a \emph{loop} of $D$. 

If $e$ is not a coloop, then, following Bouchet and Duchamp \cite{BouchetDuchamp1991}, we define the \emph{deletion} of $e$ from $D$,
written $D \backslash e$, to be
\[
D \backslash e := \bigl(E \backslash e,\; \{F \mid F \in \F \text{ and } F \subseteq E \backslash e\}\bigr).
\]
If $e$ is not a loop, then we define the \emph{contraction} of $e$, written $D/e$, to be
\[
D/e := \bigl(E \backslash e,\; \{F \backslash e \mid F \in \F \text{ and } e \in F\}\bigr).
\]
If $e$ is a loop or a coloop, then one of $D \backslash e$ and $D/e$ has already been defined, and we set the other equal to it, so that
\[
D/e = D \backslash e.
\]
Both $D \backslash e$ and $D/e$ are delta-matroids (see \cite{BouchetDuchamp1991}). 

For a delta-matroid $D$, since deletions and contractions are independent of the order (see  \cite{BouchetDuchamp1991}), we will define $D' = D \backslash X / Y$ by deleting all elements in \(X\) and contracting all elements in \(Y\) in any order.  Any delta-matroid obtainable from $D$ by a sequence of deletions and contractions is called a \emph{minor} of $D$.

\subsection{Twist, loop complementation and partial-twualities}
 Let \( D = (E, \F) \) be a set system and let \( A \subseteq E \). The \emph{twist} \cite{Bouchet1987} of \( D \) with respect to \( A \), denoted by \( D^{\ast|A} \), is the set system \( (E, \F \triangle A) \), where
    \[
    \F \triangle A := \{ X \triangle A \mid X \in \F \}.
    \]
The dual of $D$ is
\(
D^\ast=D^{\ast|E}.
\)

Let $D=(E,\F)$ be a set system and let $e\in E$. The \emph{loop complementation} \cite{Brijder2011} of $D$ at
$e$ is the set system
\(
D^{\times|e}=(E,\F'),
\)
where
\[
\F'
=
\F\triangle
\{F\cup e\mid F\in\F,\ e\notin F\}.
\]
Loop complementations on different elements commute (see \cite{Brijder2011}). Hence, for $A\subseteq E$,
the set system $D^{\times|A}$ is well-defined.

It has been shown in \cite{Brijder2011} that, for any fixed element \( e \in E \), the twist and loop complementation are involutions (i.e., of order 2). These two operations generate a group isomorphic to \( S_3 \), with the presentation:
\[
 \langle \ast, \times \mid \ast^2, \times^2, (\ast \times  )^3 \rangle.
\]
For a word $\bullet$ in $\ast$ and $\times$, and for $A\subseteq E$, we write
$D^{\bullet|A}$ for the result of applying the operation $\bullet$ to every element of
$A$. When $\bullet$ is a product of operations, the application is sequential. For instance,
\(D^{\ast \times  |A} = (D^{\ast|A})^{\times|A}\).

\begin{definition}[\cite{YanJin2024}]
   For $\bullet \in \{\ast, \times, \ast \times  , \times\ast , \ast \times  \ast \}$, the \emph{partial-$\bullet$ polynomial} of any set system $D=(E, \F)$ is defined to be the generating function
$$^{\partial}w_{D}^{\bullet}(z):=\sum_{A\subseteq E}z^{w(D^{\bullet\mid A})}$$
that enumerates all partial-$\bullet$ duals of $D$ by width.
\end{definition}

\subsection{Binary delta-matroids and types of elements}

Let $C$ be a symmetric matrix over $\GF(2)$ with rows and columns indexed by a finite
set $E$. For $A\subseteq E$, let $C[A]$ denote the principal submatrix induced by $A$.
Define
\(
D(C)=(E,\F(C)),
\)
where
\[
\F(C)=\{A\subseteq E\mid C[A]\text{ is nonsingular over }\GF(2)\}.
\]
By convention, $C[\emptyset]$ is nonsingular, so $\emptyset\in\F(C)$.

\begin{definition}[\cite{BouchetDuchamp1991}]
A delta-matroid $D$ is \emph{binary} if some twist of $D$ is isomorphic to $D(C)$ for a
symmetric matrix $C$ over $\GF(2)$.
\end{definition}

Note that binary delta-matroids are closed under deletion and contraction, and Brijder and Hoogeboom \cite{BrijderHoogeboom2013} proved that they are also closed under twists and loop complementation. In particular, all partial-twuals of a binary
delta-matroid are again binary delta-matroids.

We now recall the classification of elements in a set system by their primal and dual types. Let $D=(E,\F)$ be a proper set system and let $e\in E$.
The \emph{primal type} of $e$ is one of $p,u,t$. It can be characterized as follows:
\begin{enumerate}[label=(\roman*)]
\item $e$ has primal type $p$ in $D$ if and only if there exists $F\in\F_{\min}(D)$ such
that $e\in F$;
\item $e$ has primal type $u$ in $D$ if and only if for every
$F\in\F_{\min}(D)\cup\F_{\min+1}(D)$, one has $e\notin F$;
\item $e$ has primal type $t$ in $D$ if and only if $e\notin F$ for every
$F\in\F_{\min}(D)$, and there exists $F_1\in\F_{\min+1}(D)$ such that
$e\in F_1$.
\end{enumerate}

The \emph{dual type} of $e$ in $D$ is its primal type in the dual delta-matroid $D^\ast$.
Equivalently:
\begin{enumerate}[label=\textup{(\roman*)}]
\item $e$ has dual type $p$ in $D$ if and only if there exists $F\in\F_{\max}(D)$ such
that $e\notin F$;
\item $e$ has dual type $u$ in $D$ if and only if for every
$F\in\F_{\max}(D)\cup\F_{\max-1}(D)$, one has $e\in F$;
\item $e$ has dual type $t$ in $D$ if and only if $e\in F$ for every
$F\in\F_{\max}(D)$, and there exists $F_1\in\F_{\max-1}(D)$ such that
$e\notin F_1$.
\end{enumerate}

The \emph{type} of $e$ is the two-letter word obtained by writing its primal type first and
dual type second. Thus, for example, type $tu$ means primal type $t$ and dual type
$u$.

\subsection{Single-element width changes and interpolation terminology}

Table \ref{ta 1} \cite{YanJin2024} gives the width change under a single-element partial-twuality. In particular, for any \(\bullet \in \{\ast,\times,\ast \times  ,\times\ast ,\ast \times  \ast \}\),
\[
|w(D^{\bullet|e})-w(D)|\leq 2.
\]

\begin{table}[!t]
\centering
\caption{The difference $w(D^{\bullet|e})-w(D)$ for any \(\bullet\in 
\{\ast,
\times,
\ast \times  ,
\times\ast ,
\ast \times  \ast \}.
\)}
\label{ta 1}
\begin{tabular}{c|ccccc}
  \hline
  Type of $e$& $~~~\ast~~~$ & $~~~\times~~~$ & $~~~\ast \times  ~~~$ & $~~~\times\ast ~~~$ & $~~~\ast \times  \ast ~~~$ \\\hline
  $pp$        & $+2$      & $+1$        & $+2$            & $+2$            & $+1$ \\
  $uu$        & $-2$      & 0         & $-1$            & $-1$            & 0 \\
  $pu$        & 0       & 0         & $+1$            & 0             & $+1$ \\
  $up$        & 0       & $+1$        & 0             & $+1$            & 0 \\
  $tp$        & $+1$      & $+1$        & $+1$            & 0             & $-1$ \\
  $tu$        & $-1$      & 0         & 0             & $-2$            & $-1$ \\
  $pt$        & $+1$      & $-1$        & 0             & $+1$            & $+1$ \\
  $ut$        & $-1$      & $-1$        & $-2$            & 0             & 0 \\
  $tt$        & 0       & $-1$        & $-1$            & $-1$            & $-1$ \\
  \hline
\end{tabular}
\end{table}

For a polynomial $p(z)=\sum_{i=0}^{n}c_i z^i$, we say that $p(z)$ has a \emph{gap} of size $k$ at coefficient $c_i$ if $c_{i-1}c_{i+k}\neq 0$ but $c_i=c_{i+1}=\cdots=c_{i+k-1}=0$.
The polynomial $p(z)$ is \emph{interpolating} if it is non-zero and has no gaps.
Write $p(z)=p_e(z^2)+z\,p_o(z^2)$, where $p_e(z^2)$ and $p_o(z^2)$ consist of the even-degree and odd-degree terms of $p(z)$, respectively.
We call $p(z)$ \emph{even-interpolating} (resp.\ \emph{odd-interpolating}) if $p_e(z^2)$ (resp.\ $p_o(z^2)$) is interpolating.
An \emph{even} (resp.\ \emph{odd}) polynomial is a polynomial such that the only terms that have non-zero coefficients are the terms of even (resp.\ odd) degree.

\section{Auxiliary lemmas}\label{3}

This section collects the lemmas needed to prove our main theorem.
The following lemma shows that deletion operations do not alter the primal types of other elements.
\begin{lemma}[\cite{{ZhaoYan2026}}]\label{lem1}
Let $D=(E,\F)$ be a delta-matroid and $e_1,e_2\in E$ with $e_1\neq e_2$. 
Then the primal type of $e_2$ in $D\backslash e_1$ is the same as its primal type in $D$.
\end{lemma}
The next lemma provides a sufficient condition for a binary delta-matroid to be even based on the primal types of its elements.
\begin{lemma}[\cite{{ZhaoYan2026}}]\label{lem2}
Let $D=(E,\F)$ be a binary delta-matroid. If every element of $E$ has primal type $u$
in $D$, then $D$ is even.
\end{lemma}
\iffalse
\begin{proof}
Since every element has primal type $u$, no element appears in any feasible set of
minimum size or next-to-minimum size. Hence,
\[
\F_{\min}(D)={\emptyset},
\qquad
\F_{\min+1}(D)=\emptyset.
\]
Thus, $D$ is normal and has no feasible singleton.

Since $D$ is normal and binary, there exists a unique symmetric matrix $C$ over
$\GF(2)$ such that
\[
D=D(C).
\]
For every $v\in E$, the singleton ${v}$ is feasible if and only if the diagonal entry
$C_{vv}$ is $1$. Since there are no feasible singletons, all diagonal entries of $C$ are
zero.

Suppose, for contradiction, that $D$ is not even. Then $D$ has a feasible set
$A$ of odd cardinality. Since $A$ is feasible, the principal matrix $C[A]$ is
nonsingular. But $C[A]$ is a symmetric matrix over $\GF(2)$ with zero diagonal.
Every such matrix has even rank. Since $|A|$ is odd, its rank is strictly less than
$|A|$. Hence, $C[A]$ is singular, a contradiction. Therefore, all feasible sets have
even cardinality, and $D$ is even.
\end{proof}
\fi
Since deletion and twisting commute for distinct elements, the following lemma is immediate.
\begin{lemma}[\cite{{ChunMoffattNobleRueckriemen2019}}]\label{lem3}
Let $D$ be a delta-matroid and let $A,B\subseteq E(D)$ with $A\cap B=\emptyset$.
Then
\[
(D^{\ast|A})\backslash B=(D\backslash B)^{\ast|A}.
\]
\end{lemma}

\begin{lemma}[\cite{ChunMoffattNobleRueckriemen2019}]\label{Chu1}
Let \(D=(E,\F)\) be a delta-matroid and suppose that \(e\) has primal
type \(t\). Then, for every \(M\subseteq E\backslash e\), $M \in\F_{\min}(D)$ if and only if $M \cup  e\in\F_{\min+1}(D)$. 
\end{lemma}

We now describe how loop complementation affects primal types.

\begin{lemma}\label{lem4}
Let $D$ be a delta-matroid and let $e,f\in E(D)$.
\begin{enumerate}[label=\textup{(\alph*)}]
\item Under loop complementation at $e$, the primal type of $e$ in $D$ transforms to its primal type in $D^{\times|e}$ as follows:
$u$ and $t$ are interchanged, while $p$ remains $p$.

\item If $f\neq e$, then the primal type of $e$ in $D$ is the same as its primal type in $D^{\times|f}$.
\end{enumerate}
\end{lemma}

\begin{proof}
Let \(r=r_{\min}(D)\). First consider loop complementation at $e$. It follows from the definition of loop complementation that
\(
\F_{\min}(D^{\times|e})=\F_{\min}(D).
\)
Hence, if \(e\) has primal type $p$  in  \(D\), then \(e\) has primal type $p$ in \(D^{\times|e}\). 

Suppose $e$ has primal type $u$ in $D$. Then $e$ appears in no member of
\(
\F_{\min}(D)\cup\F_{\min+1}(D).
\)
In particular, for every \(M\in\F_{\min}(D)\), the set $M\cup e$ is not a feasible set of $D$.
Hence, by the definition of loop complementation, $M\cup e$ is a feasible set of
$D^{\times|e}$ whenever $M\in\F_{\min}(D)$. Since $\F_{\min}(D^{\times|e})=\F_{\min}(D)$,
the element $e$ appears in no member of $\F_{\min}(D^{\times|e})$, while each
$M\cup e$ (having size $r+1$) belongs to $\F_{\min+1}(D^{\times|e})$.
Thus, $e$ has primal type $t$ in $D^{\times|e}$.

Conversely, suppose $e$ has primal type $t$ in $D$. By Lemma~\ref{Chu1}, for every
$M\subseteq E\backslash e$, we have $M\in\F_{\min}(D)$ if and only if
$M\cup e\in\F_{\min+1}(D)$. In $D^{\times|e}$, each such set $M\cup e$ is removed
from the feasible family, so no feasible set in $\F_{\min+1}(D^{\times|e})$ contains
$e$. Moreover, since $\F_{\min}(D^{\times|e})=\F_{\min}(D)$ and $e$ appears in no
member of $\F_{\min}(D)$, the element $e$ does not belong to any set in
$\F_{\min}(D^{\times|e})$. Thus, $e$ has primal type $u$ in $D^{\times|e}$. This
proves part (a).

Now assume $f\neq e$. Again $\F_{\min}(D^{\times|f})=\F_{\min}(D)$, so if $e$ has
primal type $p$ in $D$, then $e$ also has primal type $p$ in $D^{\times|f}$.

Suppose $e$ does not have primal type $p$. Then no feasible set in $\F_{\min}(D)$
contains $e$. Let $S\subseteq E$ with $|S|=r+1$ and $e\in S$. If $f\notin S$, then
loop complementation at $f$ does not affect $S$, so $S\in\F_{\min+1}(D)$ if and
only if $S\in\F_{\min+1}(D^{\times|f})$. If $f\in S$, then
$|S\backslash f|=r$ and $e\in S\backslash f$. Since $e$ belongs to no set in
$\F_{\min}(D)$, the set $S\backslash f$ is not feasible in $D$, and consequently
loop complementation at $f$ leaves the feasibility of $S$ unchanged. Hence,
$S\in\F_{\min+1}(D)$ if and only if $S\in\F_{\min+1}(D^{\times|f})$.
Thus, the members of $\F_{\min+1}$ that contain $e$ are the same in $D$ and
$D^{\times|f}$.

Hence, the primal type of $e$ in $D^{\times|f}$ is the same as its primal type in $D$.
This proves part (b).
\end{proof}

\begin{remark}
The delta-matroid assumption in Lemma~\ref{lem4}(a) is needed only for the
case \(t\mapsto u\). The proof of this case uses Lemma~\ref{Chu1}, which is valid
for delta-matroids but not for arbitrary proper set systems. For example,
let
\[
    E=\{1,2,3\},
    \qquad
    \F=\bigl\{\{2\},\{1,3\}\bigr\}.
\]
Then \(1\) has primal type \(t\) in \(D=(E,\F)\). However,
\[
    \F(D^{\times|1})
    =
    \bigl\{\{2\},\{1,2\},\{1,3\}\bigr\},
\]
so \(1\) still has primal type \(t\) in \(D^{\times|1}\), rather than primal type
\(u\). In contrast, the cases \(p\mapsto p\) and \(u\mapsto t\) hold for
arbitrary proper set systems. On the other hand, Lemma~\ref{lem4}(b) does not
require the delta-matroid assumption.
\end{remark}

\begin{lemma}\label{lem5}
Let $D$ be a binary delta-matroid and let $S\subseteq E(D)$. Suppose every element of
$S$ has primal type $u$ in $D$. Let $T\subseteq S$ and let $e\in S\backslash T$. Then $e$ does
not have primal type $t$ in $D^{\ast|T}$.
\end{lemma}

\begin{proof}
For convenience, set $S^c = E(D)\backslash S$.  By Lemma~\ref{lem1}, deleting elements
of $S^c$ does not change the primal type of any element of $S$. Hence, every element
of $S$ has primal type $u$ in
\(D\backslash S^c.\)
Since binary delta-matroids are closed under deletion, $D\backslash S^c$ is binary. By
Lemma~\ref{lem2}, the delta-matroid $D\backslash S^c$ is even.
Twisting preserves evenness, so
\((D\backslash S^c)^{\ast|T}\)
is even. Since $T\subseteq S$, Lemma~\ref{lem3} gives
\[
(D^{\ast|T})\backslash S^c
=
(D\backslash S^c)^{\ast|T}.
\]
Thus, $(D^{\ast|T})\backslash S^c$ is even.

Suppose, for contradiction, that $e$ has primal type $t$ in $D^{\ast|T}$. By
Lemma~\ref{lem1}, $e$ still has primal type $t$ in \((D^{\ast|T})\backslash S^c.\) But an even delta-matroid has no element of primal type $t$, a contradiction.
Therefore, $e$ cannot have primal type $t$ in $D^{\ast|T}$.
\end{proof}

We now prove the main lemma. Note that
\(
(\ast\times)^{-1}=\times \ast,\)
and 
\((\times \ast)^{-1}=\ast \times.
\)
\begin{lemma}\label{lem6}
Let $D$ be a binary delta-matroid and let $X,Y\subseteq E(D)$ be disjoint subsets.
Suppose that in $D$ every element of $X$ has type $tu$ and every element of $Y$ has
type $ut$. Let \( X_0\subseteq X,\) and  \(Y_0\subseteq Y,\)
and define \(D_0=(D^{\times\ast| X_0}{})^{\ast \times|  Y_0}.\)
Then the following hold.
\begin{enumerate}[label=\textup{(\roman*)}]
\item If $x\in X\backslash X_0$, then in $D_0$ the type of $x$ belongs to
\(
\{pp,pu,tp,tu\}.
\)
Consequently,
\(
w(D_0^{\times \ast | x})-w(D_0)
\)
is even.

\item If $y\in Y\backslash Y_0$, then in $D_0$ the type of $y$ belongs to
\( \{pp,pt,up,ut\}. \)
Consequently,
\( w(D_0^{\ast \times |  y})-w(D_0)\)
is even.

\end{enumerate}
\end{lemma}

\begin{proof}
Let \(S=X\cup Y,\) and \(T=X_0\cup Y_0.\)
Since $X_0\subseteq X$ and $Y_0\subseteq Y$, the sets $X_0$ and $Y_0$ are disjoint. We first prove statement (i). Let \(x\in X\backslash X_0.\) We shall show that in $D_0$ the primal type of $x$ is not $u$, and the dual type of
$x$ is not $t$. These two exclusions imply that the type of $x$ belongs to \(\{pp,pu,tp,tu\}.\)

\medskip
\textbf{Step 1:} the primal type of $x$ in $D_0$ is not $u$.

Since every element of $X$ has type $tu$ in $D$, every element of $X$ has primal type
$t$ in $D$. Every element of $Y$ has type $ut$ in $D$, so every element of $Y$ has primal
type $u$ in $D$.  Define  $\widetilde D=D^{\times|X}.$
By Lemma~\ref{lem4}, every element of \(S=X\cup Y\) has primal type $u$ in $\widetilde D$.
Since $x\in X\backslash X_0$, we have \(x\notin T.\)
Applying Lemma~\ref{lem5} to $\widetilde D$, $S$, $T$, and $x$, we
obtain that $x$ does not have primal type $t$ in
\(\widetilde D^{\ast|T}.\)
Therefore, the primal type of $x$ in $\widetilde D^{\ast|T}$ is either $p$ or $u$.
Now compare $\widetilde D^{\ast|T}$ with $D_0$. Since
\[D_0=(D^{\times\ast |X_0}{})^{\ast \times |Y_0}\]
and \[\widetilde D^{\ast|T}=(D^{\times|X}{})^{\ast|X_0\cup Y_0},\]
we have
\[(\widetilde D^{\ast|T}{})^{\times|(X\backslash X_0)\cup Y_0}=((D^{\times|X}{})^{\ast|X_0\cup Y_0}{})^{\times|(X\backslash X_0)\cup Y_0}=D_0.\]
By performing loop complementations on $\widetilde D^{\ast|T}$ at the subset \((X\backslash X_0)\cup Y_0\), we obtain \(D_0\).
Since \(x\in (X\backslash X_0)\cup Y_0\) and the primal type of $x$ in $\widetilde D^{\ast|T}$ is either $p$ or $u$, by Lemma \ref{lem4}(a), it follows that the primal type of $x$ in $D_0$ is either $p$ or $t$. Equivalently, it is not $u$.

\medskip
\textbf{Step 2:} the dual type of $x$ in $D_0$ is not $t$.

The dual type of $x$ in $D_0$ is the primal type of $x$ in $D_0^\ast$. We therefore
consider the dual delta-matroid $D^\ast$.
In $D$, every element of $X$ has type $tu$, so in $D^\ast$ every element of $X$ has
primal type $u$. Similarly, every element of $Y$ has type $ut$ in $D$, so in $D^\ast$
every element of $Y$ has primal type $t$.
Define $\widehat D=(D^\ast)^{\times|Y}.$
By Lemma~\ref{lem4}, every element of $S=X\cup Y$ has primal type $u$ in $\widehat D$. Since $x\notin T$, Lemma~\ref{lem5}
implies that $x$ does not have primal type $t$ in
\(\widehat D^{\ast|T}.\)
We now compare $D_0^\ast$ with $\widehat D^{\ast|T}$, 
where
\[
D_0^\ast
=((D^{\times\ast|X_0}){}^{\ast \times  |Y_0}){}^{\ast|E(D)}=
((D^\ast)^{\ast \times  |X_0}){}^{\times\ast |Y_0},
\]
and
\[
\widehat D^{\ast|T}
=
((D^\ast)^{\times|Y}){}^{\ast|X_0\cup Y_0}.
\]
Thus,
\[
(\widehat D^{\ast|T})^{\times\,|\,X_0\cup (Y\backslash Y_0)}
=(((D^\ast)^{\times|Y})^{\ast|X_0\cup Y_0})^{\times\,|\,X_0\cup (Y\backslash Y_0)}
=D_0^*.
\]
Since \(x\notin X_0\cup (Y\backslash Y_0)\) and the primal type of $x$ in $\widehat D^{\ast|T}$ is either $p$ or $u$, by Lemma~\ref{lem4}(b), it follows that the primal type of $x$ in $D_0^\ast$ is either $p$ or $u$. 
Hence, $x$ does not have primal type $t$ in $D_0^\ast$. Equivalently, $x$ does not
have dual type $t$ in $D_0$.

Combining Steps 1 and 2, the primal type of $x$ in $D_0$ lies in $\{p,t\}$, and
the dual type of $x$ lies in $\{p,u\}$. Hence, the type of $x$ in $D_0$ is one of
\(
\{pp, pu, tp, tu\}.
\)
Consulting Table \ref{ta 1} for $\times\ast$, the corresponding width
changes are
\(\{2, 0, 0, -2\}.\)
Thus, \(w(D_0^{\times\ast |x})-w(D_0)\)
is even. This proves statement (i).

We now prove statement (ii). Let
\(y\in Y\backslash Y_0.\)
We shall show that in $D_0$ the primal type of $y$ is not $t$, and the dual type of
$y$ in $D_0$ is not $u$. These two exclusions imply that the type of $y$ in $D_0$ belongs to \(\{pp,pt,up,ut\}.\)

\medskip
\textbf{Step 3:} the primal type of $y$ in $D_0$ is not $t$.

As above, set \(\widetilde D=D^{\times|X}.\)
We have already shown that every element of $S=X\cup Y$ has primal type $u$ in
$\widetilde D$. Since $y\in Y\backslash Y_0$, we have
\(y\notin T.\)
By Lemma~\ref{lem5}, $y$ does not have primal type $t$ in \(\widetilde D^{\ast|T}.\) Since 
\[D_0= (\widetilde D^{\ast|T}){}^{\times|(X\backslash X_0)\cup Y_0},\] 
and $y\in Y\backslash Y_0$ implies $y\notin (X\backslash X_0)\cup Y_0$, the loop
complementations that produce $D_0$ do not involve $y$. By Lemma~\ref{lem4}(b),
such loop complementations do not change the primal type of $y$. Therefore, $y$ does
not have primal type $t$ in $D_0$.

\medskip
\textbf{Step 4:} the dual type of $y$ in $D_0$ is not $u$.

We again consider $D^\ast$ and use
\(\widehat D=(D^\ast)^{\times|Y}.\)
As before, every element of $S=X\cup Y$ has primal type $u$ in $\widehat D$.
Since $y\notin T$, Lemma~\ref{lem5} implies that $y$ does not have
primal type $t$ in \(\widehat D^{\ast|T}.\)
From the identity
\[
D_0^\ast = (\widehat D^{\ast|T})^{\times|X_0\cup (Y\backslash Y_0)},
\]
we see that the loop complementation acts on the subset $X_0\cup (Y\backslash Y_0)$,
which contains $y$. As the primal type of $y$ in
$\widehat D^{\ast|T}$ is either $p$ or $u$, Lemma~\ref{lem4}(a) implies that its
primal type in $D_0^\ast$ is either $p$ or $t$. Equivalently, the dual type of $y$
in $D_0$ is not $u$.

Combining Steps 3 and 4, the primal type of $y$ in $D_0$ lies in $\{p,u\}$, and
the dual type of $y$ lies in $\{p,t\}$. Hence, the type of $y$ in $D_0$ is one of \( \{pp, pt, up, ut\}.\)
Consulting Table \ref{ta 1} for $\ast \times  $, the corresponding width
changes are
\( \{2, 0, 0, -2\}. \)
Thus, \(w(D_0^{\ast \times | y})-w(D_0)\) is even. This proves statement (ii), and the lemma follows.
\end{proof}

\section{Main Theorems}\label{4}
For a delta-matroid $D$, define the width sets
\[
W_{\ast\times}(D)=\{w(D^{\ast\times|A})\mid A\subseteq E(D)\},
\quad
W_{\times\ast}(D)=\{w(D^{\times\ast|A})\mid A\subseteq E(D)\}.
\]

\begin{theorem}\label{thm1}
Let $D$ be a binary delta-matroid. Suppose
\(k,k+1\in W_{\ast\times}(D).\)
Then for every integer $m$ satisfying
\[
k+2\leq m\leq \max W_{\ast\times}(D),
\]
we have \(m\in W_{\ast\times}(D).\)
\end{theorem}

\begin{proof}
Let \(M=\max W_{\ast\times}(D).\) If $M\leq k+1$, there is nothing to prove. Assume $M\geq k+2$. 
It suffices to show $k+2\in W_{\ast\times}(D)$. Once $k+2$ is obtained, applying the same reasoning to the consecutive pair $k+1,k+2$ yields $k+3$, and proceeding inductively, we obtain all integers up to $M$.

Choose $A\subseteq E(D)$ such that
\(w(D^{\ast \times |  A})=k.\)
Suppose, for contradiction, that \(k+2\notin W_{\ast\times}(D).\)
Since $M\geq k+2$, there exists $C\subseteq E(D)$ with
\(w(D^{\ast \times |  C})\geq k+3.\)
Move from $A$ to $C$ by changing one element at a time. At each step the width changes
by at most $2$. Since the width $k+2$ is forbidden, the sequence of widths must attain
the value $k+3$. Hence, there exists $B\subseteq E(D)$ such that
$w(D^{\ast \times |  B})=k+3.$
Choose such a set $B$ so that \(|A\triangle B|\) is minimal. Set $Q=D^{\ast \times |  B}$ and let \[X=B\backslash A \quad \text{and} \quad Y=A\backslash B.\]
We first determine the types of elements of $X$ and $Y$ in $Q$. Let $e\in X$. Then $e\in B$ but $e\notin A$. We have 
\[
D^{\ast \times |  B\triangle e}= (D^{\ast \times |  B}){}^{\times \ast | e}= Q^{\times \ast | e}.
\]
Moreover,
\[
|A\triangle(B\triangle e)|<|A\triangle B|.
\]
By the minimality of \(|A\triangle B|\), we have
\[
w(Q^{\times \ast | e})\neq k+3.
\]
By our contradiction assumption, we also have
\[
w(Q^{\times \ast | e})\neq k+2.
\]

Since a single $\times\ast$-operation changes width by at most $2$, and $w(Q)=k+3$, the only
possible values for $w(Q^{\times \ast | e})$ are
\[
\{k+1,\ k+2,\ k+3,\ k+4,\ k+5\}.
\]
The values $k+2$ and $k+3$ have just been excluded.
Suppose that \(w(Q^{\times \ast | e})\geq k+4.\) 
Let 
\[A\triangle(B\triangle e)=\{e_1,e_2,\dots, e_m\}.\]
Since a single $\ast\times$-operation or $\times\ast$-operation changes width by at most $2$ and \(w(D^{\ast \times |  A})=k; w(D^{\ast \times |  B\triangle e})= w(D^{\ast \times |  A\triangle \{e_1,e_2,\dots, e_m\}})\geq k+ 4\),
\[k+3\in \{w(D^{\ast \times|A}),w(D^{\ast \times|A\triangle e_1}),w(D^{\ast \times |  A\triangle \{e_1,e_2\}}), \dots, w(D^{\ast \times |  A\triangle \{e_1,e_2,\dots, e_m\}})\}.\]
This means that there exists a subset $B'$ such that
\[
w(D^{\ast \times |  B'})=k+3
\]
and
\[
|A\triangle B'|<|A\triangle B|,
\]
contradicting the minimality of \(|A\triangle B|\). Thus, $w(Q^{\times\ast|e})$ cannot be $k+2$, $k+3$, $k+4$, or $k+5$. The only remaining possibility is $k+1$. Therefore,
\[
w(Q^{\times\ast|e}) = k+1.
\]
Hence,
\[
w(Q^{\times \ast |  e})-w(Q)=-2.
\]
By Table \ref{ta 1}, for $\times\ast $ the width change $-2$ occurs only for type $tu$. Thus, every element of $X$ has type $tu$ in $Q$.

Now let $e\in Y$. Then $e\notin B$ but $e\in A$. Then 
\[
D^{\ast \times |  B\triangle e}=(D^{\ast \times |B}) {}^{\ast \times |  e}=Q^{\ast \times |  e}.
\]
The same argument as above gives
\[
w(Q^{\ast \times |  e})=k+1.
\]
Hence,
\[
w(Q^{\ast \times |  e})-w(Q)=-2.
\]
By Table \ref{ta 1}, for $\ast \times $ the width change $-2$ occurs only for
type $ut$. Thus, every element of $Y$ has type $ut$ in $Q$.

Consequently, we have shown that if \(x\in X\), then \(x\) has type \(tu\) in \(Q\), while if \(y\in Y\), then \(y\) has type \(ut\) in \(Q.\)

Now start from $Q=D^{\ast \times |  B}$ and transform it into $D^{\ast \times |  A}$. To do this, apply \(\times\ast \) to every element of $X=B\backslash A$, and apply \(\ast \times  \) to every element of $Y=A\backslash B$. Choose any order for these operations. At any intermediate stage, suppose that the
already processed subsets are
\(X_0\subseteq X,\) and \(Y_0\subseteq Y.\) The current binary delta-matroid is \( (Q^{\times \ast | X_0}){}^{\ast \times |  Y_0}.\)
If the next element is in $X\backslash X_0$, then by Lemma~\ref{lem6}, applying $\times\ast$ changes the width by an even number.
If the next element is in $Y\backslash Y_0$, the same
lemma says that applying $\ast\times$ changes the width by an even number.

Therefore, every step in the transformation from $Q$ to $D^{\ast \times |  A}$ changes the width
by an even number. Hence, the total width change must be even. But
\[
w(D^{\ast \times |  A})-w(Q)
=
k-(k+3)
=
-3,
\]
which is odd. This contradiction shows that our assumption was false. Therefore,
\[
k+2\in W_{\ast\times}(D).
\]

As explained at the beginning of the proof, by induction we obtain every integer \(m\) with \(k+2\leq m\leq \max W_{\ast\times}(D).\)
The theorem follows.
\end{proof}

The $\times\ast $ case follows from the $\ast \times$ case by duality.

\begin{lemma}[\cite{YanJin2024}]\label{lem7}
Let $D=(E, \F)$ be a set system. Then for any $\bullet\in \{\ast,\times,\ast\times,\times\ast,\ast\times\ast\}$,
$$^{\partial}w_{D}^{\ast\bullet\ast}(z)={^{\partial}w_{D^{\ast}}^{\bullet}(z)}.$$
\end{lemma}

\begin{theorem}\label{thm2}
Let $D$ be a binary delta-matroid. Suppose
\(k,k+1\in W_{\times\ast}(D).\)
Then for every integer $m$ satisfying
\[
k+2\leq m\leq \max W_{\times\ast}(D),
\]
we have \(m\in W_{\times\ast}(D).\)
\end{theorem}

\begin{proof}
Since $\ast(\ast\times)\ast=\times\ast$, Lemma~\ref{lem7} gives
\[
^{\partial}w_{D}^{\times\ast}(z) = {^{\partial}w_{D^\ast}^{\ast\times}(z)}.
\]
Thus, $W_{\times\ast}(D)=W_{\ast\times}(D^\ast)$. Since $D$ is binary, its dual
$D^\ast$ is also binary. Then the result follows immediately from
Theorem~\ref{thm1} applied to $D^\ast$.
\end{proof}

We now prove the main interpolation theorem.

\begin{proof}[Proof of Theorem \ref{thm}]
For $\bullet = \ast$, this is the twist polynomial interpolation theorem for binary delta-matroids given in \cite{ZhaoYan2026}.

For $\bullet=\times$ and $\bullet=\ast \times  \ast $, Table \ref{ta 1} shows that applying either $\times$ or $\ast\times\ast$ to a single element changes the width by at most $1$. Thus, both ${^{\partial}w_{D}^{\times}(z)}$ and ${^{\partial}w_{D}^{\ast \times  \ast }(z)}$ are interpolating.

It remains to handle \(\bullet=\ast \times  \) and 
\(\bullet=\times\ast .\)
We treat both cases simultaneously. Let
\[
W_\bullet(D)
=
\{w(D^{\bullet|A})\mid A \subseteq E(D)\}.
\]
List its elements in increasing order:
\[
w_1<w_2<\cdots<w_s.
\]
By Table \ref{ta 1}, changing one element changes the width by at most $2$.
Therefore, for every $i$,
\[
w_{i+1}-w_i\leq 2.
\]
If all elements of $W_\bullet(D)$ have the same parity, then
${^{\partial}w_{D}^{\bullet}(z)}$ is either an even polynomial or an odd polynomial.

Now suppose that $W_\bullet(D)$ contains both parities. Then there exists a smallest index
$i$ such that $w_i$ and $w_{i+1}$ have opposite parity. Since
\[
1\leq w_{i+1}-w_i\leq 2
\]
and the difference is odd, we must have
\[
w_{i+1}=w_i+1.
\]

If $\bullet=\ast\times$, Theorem~\ref{thm1} implies that every integer from $w_i+2$ up
to $\max W_{\ast\times}(D)$ belongs to $W_{\ast\times}(D)$. Hence, every integer from $w_i$ to the maximum
belongs to $W_{\ast\times}(D)$.

If $\bullet=\times\ast$, the same conclusion follows from Theorem~\ref{thm2}.

Thus, ${^{\partial}w_{D}^{\bullet}(z)}$ is both even-interpolating and odd-interpolating.
The theorem follows.
\end{proof}

\section{Application to ribbon graphs}\label{5}
We now transfer the results on delta-matroids to ribbon graphs.

\subsection{Ribbon Graphs}
	A ribbon graph arises naturally from a cellularly embedded graph by taking a small neighborhood for each vertex and edge, and is formally defined as follows:
	\begin{definition}[\cite{Bollobas2002}]
		\normalfont
		A \emph{ribbon graph} \( G = (V(G), E(G)) \) is a surface with boundary, represented as the union of two sets of topological discs: a set \( V(G) \) of vertices and a set \( E(G) \) of edges, satisfying the following properties:
		\begin{enumerate}[label=\textup{(\alph*)}]
			\item  The vertices and edges intersect in disjoint line segments.
			\item   Each such line segment lies on the boundary of exactly one vertex and exactly one edge.
			\item   Every edge contains exactly two such line segments.
		\end{enumerate}
	\end{definition}

Let \( G \) be a ribbon graph. A \emph{ribbon subgraph} of \( G \) is obtained by deleting some vertices and edges from \( G \). A \emph{spanning ribbon subgraph} is obtained by deleting edges only. A \emph{quasi-tree} \( Q \) is a connected ribbon graph with exactly one boundary component. For a connected ribbon graph \( G \), a \emph{spanning quasi-tree} \( Q \) of \( G \) is a spanning ribbon subgraph with exactly one boundary component. For a disconnected ribbon graph $G$, we say that a ribbon graph $Q$ is a spanning quasi-tree of $G$ if $k(Q) = k(G)$ (where $k(G)$ denotes the number of connected components of $G$) and each connected component of $Q$ is a spanning quasi-tree of the corresponding connected component of $G$.

\subsection{Partial dual, partial Petrial and partial-twualities}

The partial dual and the partial Petrial of a ribbon graph $G$ with respect to a subset $A\subseteq E(G)$ are defined as follows.
\begin{definition} [\cite{Chmutov2009}]
\normalfont
For a ribbon graph $G$ and $A\subseteq E(G)$, the partial dual $G^{\delta|A}$ of $G$ with respect to $A$ is a ribbon graph obtained from $G$ by gluing a disc to $G$ along each boundary component of the spanning ribbon subgraph $(V(G), A)$ (such discs will be the vertex-discs of $G^{\delta|A}$), removing the interiors of all the vertex-discs of $G$ and keeping its edge-ribbons unchanged. 
\end{definition}
In particular, the geometric dual of $G$, denoted $G^*$, is the partial dual with respect to $E(G)$, that is, $G^* = G^{\delta| E(G)}$.

\begin{definition} [\cite{EllisMonaghanMoffatt2012}]
\normalfont
For a ribbon graph $G$ and $A\subseteq E(G)$, the partial Petrial $G^{\tau|A}$ of $G$ with respect to $A$ is a ribbon graph obtained from $G$ by adding a half-twist to each of the edges in $A$.
\end{definition}
Let
\[
    \mathcal{R}=\langle \delta,\tau \mid \delta^2,\tau^2,(\delta\tau)^3\rangle
\]
be the ribbon group generated by partial duality $\delta$ and partial Petriality
$\tau$ (see \cite{Wilson1979}). Its five non-identity twuality operations are
\( \{\delta,\tau,\delta\tau,\tau\delta, \delta\tau\delta\}.\)
On the set system side, twist and loop complementation generate the group
\[
    \mathcal{B}=\langle \ast,\times \mid \ast^2,\times^2,(\ast \times  )^3\rangle .
\]
We use the standard isomorphism
\[
    \eta:\mathcal{R}\longrightarrow \mathcal{B},
    \qquad
    \eta(\delta)=\ast,\qquad
    \eta(\tau)=\times.
\]
Thus,
$\eta(\delta)=\ast,
\eta(\tau)=\times,
\eta(\delta\tau)=\ast \times,\eta(\tau\delta)=\times\ast,\eta(\delta\tau\delta)=\ast\times\ast.$

\subsection{ Ribbon-graphic delta-matroid and partial-$\circ$ polynomial}
Let $G = (V(G), E(G))$ be a ribbon graph, and define
\[
\F(G) := \{ F \subseteq E(G) \mid F \text{ is the edge set of a spanning quasi-tree of } G \}.
\]
Then the set system $D(G) := (E(G), \F(G))$ is a binary delta-matroid \cite{ChunMoffattNobleRueckriemen2019}.

For any twuality operation $\circ \in \{\delta,\tau,\delta\tau,\tau\delta,\delta\tau\delta\}$, 
we have the following standard compatibilities (see \cite{ChunMoffattNobleRueckriemen2019, Chun2019F}):
\[
    D(G^{\circ|A})
    =
    D(G)^{\eta(\circ)|A}
\]for every $A\subseteq E(G)$,
and
\[
    \varepsilon(G)=w(D(G)),
\]
where $\varepsilon(G)$ denotes the Euler genus of $G$.
\begin{definition}[\cite{GrossMansourTucker2021}]
For a ribbon graph $G$ and
\(
    \circ\in\{\delta,\tau,\delta\tau,\tau\delta,\delta\tau\delta\},
\)
the partial-$\circ$ polynomial is
\[
   {^{\partial}\varepsilon_{G}^{\circ}(z)}
    =\sum_{A\subseteq E(G)} z^{\varepsilon(G^{\circ|A})}.
\]
\end{definition}

\begin{proposition}[\cite{YanJin2024}]\label{pro01}
Let $G=(V(G), E(G))$ be a ribbon graph and \(
    \circ\in\{\delta,\tau,\delta\tau,\tau\delta,\delta\tau\delta\}
\). Then $$^{\partial}w_{D(G)}^{\eta(\circ)}(z)={^{\partial}\varepsilon_{G}^{\circ}(z)}.$$
\end{proposition}

\subsection{The interpolation theorem for ribbon graphs}
We first state the corresponding Theorem \ref{thm1} for ribbon graphs. For
\(
    \circ\in\{\delta,\tau,\delta\tau,\tau\delta,\delta\tau\delta\},
\)
define the partial-$\circ$ Euler-genus set of $G$ by
\[
    \Gamma_{\circ}(G)
    =
    \{\varepsilon(G^{\circ|A})\mid A\subseteq E(G)\}.
\]
Equivalently,
\[
    \Gamma_{\circ}(G)
    =
    \{w(D(G)^{\eta(\circ)|A})\mid A\subseteq E(G)\}.
\]

\begin{corollary}\label{cor:ribbon-hole-deltatau}
Let $G$ be a ribbon graph and \(\circ \in \{\delta,\tau,\delta\tau,\tau\delta,\delta\tau\delta\}\). If
\(
    k,k+1\in \Gamma_{\circ}(G),
\)
then for every integer \(m \) with
\[
    k+2\le m\le \max \Gamma_{\circ}(G),
\]
we have
\(
    m\in \Gamma_{\circ}(G).
\)
\end{corollary}

\begin{proof}
If \(\circ\in \{\tau, \delta\tau\delta\}\), we have \( \eta(\tau)=\times\) and \( \eta(\delta\tau\delta)=\ast\times\ast \).
Applying either $\times$ or $\ast\times\ast$ to a single element changes the width of $D(G)$ by at most $1$, so
\(
    \Gamma_{\circ}(G) = \{\varepsilon(G^{\circ|A})\mid A\subseteq E(G)\}
\)
constitutes a contiguous set of integers. Hence, the result holds.

If \(\circ=\delta\), the result holds by \cite{ZhaoYan2026}.

If \(\circ\in \{\tau\delta, \delta\tau\}\), since \( \eta(\tau\delta)=\times\ast \) and \( \eta(\delta\tau)=\ast \times  \) and
\(
    \Gamma_{\circ}(G)
    =
    W_{\eta(\circ)}(D(G)),
\)
the result therefore follows directly from Theorems~\ref{thm1} and~\ref{thm2}.
\end{proof}

We now state the ribbon graph counterpart of Theorem \ref{thm}.

\begin{corollary}
\label{cor:ribbon-main}
Let $G$ be a ribbon graph and let
\(
    \circ\in\{\delta,\tau,\delta\tau,\tau\delta,\delta\tau\delta\}.
\)
Then the partial-$\circ$ polynomial
\(^{\partial} \varepsilon_{G}^{\circ}(z)\)
is either an even polynomial, an odd polynomial, or both even-interpolating and
odd-interpolating.
\end{corollary}

\begin{proof}
Since $D(G)$ is a binary delta-matroid, Theorem~\ref{thm} applies to
\(D(G)\) and to every operation
\[
    \eta(\circ)\in\{\ast,\times,\ast \times  ,\times\ast ,\ast \times  \ast \}.
\]
Hence, \( ^{\partial} w_{D(G)}^{\eta(\circ)}(z)\)
is either an even polynomial, an odd polynomial, or both even-interpolating and
odd-interpolating. The same conclusion holds for
\(^{\partial} \varepsilon_{G}^{\circ}(z)\) by Proposition \ref{pro01}.
\end{proof}

\section{Non-binary examples}\label{6}

In this section we construct explicit examples showing that the binary hypothesis in
Theorem~\ref{thm} is essential. We give two kinds of obstructions. The
first shows that the twist case $\bullet=\ast$ cannot be extended to arbitrary
delta-matroids.  The second one shows that the same is true for the non-involutive
operations $\ast \times  $ and $\times\ast $.
\begin{lemma}[\cite{BouchetDuchamp1991}]\label{thm4}
A delta-matroid is binary if and only if it has no minor isomorphic to one of the
following delta-matroids:
\begin{itemize}
    \item \(D_1=(\{1,2,3\}, \{\emptyset, \{1,2\},\{2,3\},\{1,3\},\{1,2,3\}\})\);
    \item \(D_2=(\{1,2,3\}, \{\emptyset, \{1\},\{2\},\{3\},\{1,2\},\{1,3\},\{2,3\}\})\);
    \item \(D_3=(\{1,2,3\}, \{\emptyset, \{2\},\{3\},\{1,2\},\{1,3\},\{1,2,3\}\})\);
    \item \(D_4=(\{1,2,3,4\}, \{\emptyset, \{1,2\},\{1,3\},\{1,4\},\{2,3\},\{2,4\},\{3,4\}\})\);
    \item \(D_5=(\{1,2,3,4\}, \{\emptyset, \{1,2\},\{2,3\},\{3,4\},\{1,4\},\{1,2,3,4\}\})\).
\end{itemize}
\end{lemma}

\subsection{A non-binary example for the twist polynomial}

Let \(E=\{1,2,\ldots,8\},\)
and define
\[
D=(E,\F),
\qquad
\F=\{F\subseteq E\mid |F|\in\{0,1,3,4\}\}.
\]
We first verify that $D$ is a delta-matroid. Let $X,Y\in\F$ and let $u\in X\triangle Y$. 

If $u\in X\backslash Y$, then $|X\triangle u|=|X|-1$.  If $|X|\neq 3$, this is immediately \(X\triangle u\in \F\).
When $|X|=3$, if $Y\backslash X\neq\emptyset$, choose $v\in Y\backslash X$.  Then
$X\triangle\{u,v\}$ has the same size as $X$, namely $3$, and is feasible.  If
$Y\backslash X=\emptyset$, then $Y\subseteq X$.  Since $|Y|\in\{0,1,3,4\}$ and $Y\neq X$,
we have $|Y|=0$ or $1$.  Hence, $X\backslash Y$ has at least two elements, and we may choose
$v\in X\backslash Y$ with $v\neq u$.  Then \(|X\triangle\{u,v\}|=1,\) so the resulting set is feasible.

If $u\in Y\backslash X$, then $|X\triangle u|=|X|+1$. If $|X|\neq 1$ and $|X|\neq 4$, this is immediately \(X\triangle u\in \F\).
When $|X|=1$, if $Y\backslash X$ contains another element $v\neq u$, then
\(|X\triangle\{u,v\}|=3,\) so the resulting set is feasible.  If $Y\backslash X=u$, then either $Y=X\cup u$, which
would have size $2$ and hence be impossible, or $X$ and $Y$ are distinct singletons.
In the latter case, set $Y = u$ and $X = v$. Then we have
\(X \triangle \{u,v\} = u = Y \in \mathcal F.\)
When $|X|=4$, since $Y$ is feasible and has size at most $4$, the condition
$u\in Y\backslash X$ forces $X\backslash Y\neq\emptyset$.  Choosing $v\in X\backslash Y$ gives \(|X\triangle\{u,v\}|=4,\) and hence \(X\triangle\{u,v\}\) is a feasible set. Therefore, the symmetric exchange axiom holds,
and $D$ is a delta-matroid.

Next we show that $D$ is not binary. Since \(D/1\backslash \{5,6,7,8\} \) is isomorphic to \(D_1\), by Lemma~\ref{thm4}, $D$ is not binary.

We now compute its twist polynomial. Since the feasible sets of $D$ are determined
only by their cardinalities, the width of $D^{\ast|A}$ depends only on \(a=|A|.\) That is, \(w(D^{\ast|A})=w(D^{\ast|B})\) whenever \(|A|=|B|.\)
A direct calculation gives
\[
\begin{array}{c|ccccccccc}
a&0&1&2&3&4&5&6&7&8\\
\hline
w(D^{\ast|A})&4&5&5&7&8&7&5&5&4
\end{array}
\]
Therefore,
\[
\begin{aligned}
{^{\partial}w_{D}^{\ast}(z)}
&=
\sum_{A\subseteq E} z^{w(D^{\ast|A})}
&=
2z^4+72z^5+112z^7+70z^8.
\end{aligned}
\]

This polynomial contains nonzero terms in both even and odd degrees, but it is
not both even-interpolating and odd-interpolating.
This example shows that the twist polynomial interpolation theorem for binary
delta-matroids cannot be extended to arbitrary delta-matroids.

\subsection{Non-binary examples for the partial-\(\ast\times\) and partial-\(\times\ast\) polynomial}

We now give a second example showing that the binary hypothesis is also essential for
the two non-involutive operations $\ast \times  $ and $\times\ast $. Let
\[
E=\{1,2,\ldots,7\},
\]
and define
\[
D=(E,\F),
\qquad
\F=\{F\subseteq E \mid |F|\leq 2\}.
\]
This is a delta-matroid. Indeed, let $X,Y\in\F$ and let $u\in X\triangle Y$. If $u\in X\backslash Y$, choose $v=u$. Since $|X|\leq 2$, we have $|X\triangle u|\leq 2$, so $X\triangle u$ is a feasible set. If $u \in Y \backslash X$ and $|X| \leq 1$, taking $v = u$ yields that $X \triangle u$ is a feasible set.  Finally, if $u\in Y\backslash X$ and $|X|=2$,
then either there exists $v\in X\backslash Y$, in which case
\(
|X\triangle\{u,v\}|=2,
\)
or $X\subseteq Y$, implying $|Y|\geq 3$, a contradiction. Hence, the
symmetric exchange axiom holds.

Since \(D\) has a minor isomorphic to \(D_2\), by Lemma~\ref{thm4}, $D$ is not binary. For this delta-matroid, a direct computation gives
\[
{^{\partial}w_{D}^{\times\ast }(z)}
=
8z^2+21z^3+35z^5+63z^6+z^7.
\]

This polynomial contains nonzero terms in both even and odd degrees, but it is
not both even-interpolating and odd-interpolating.
\[
{^{\partial}w_{D}^{\ast\bullet\ast}(z)}={^{\partial}w_{D^\ast}^{\bullet}(z)},
\]
taking $\bullet=\ast \times  $ gives
\[
{^{\partial}w_{D}^{\times\ast }(z)}={^{\partial}w_{D^\ast}^{\ast \times}(z)},
\]
Hence, the dual delta-matroid $D^\ast$ gives a corresponding non-binary example for
$\ast\times$.

These two examples show that
the binary condition is not merely a technical assumption
in the proof. This condition is already necessary for the twist operation, and it remains necessary for the non-involutive partial-twualities $\ast\times$ and $\times\ast$.

\section*{Acknowledgements}
This work is supported by NSFC (Nos. 12471326, 12571379).

\end{document}